\documentclass{article}

\RequirePackage[margin=1in]{geometry}

\usepackage{
    amsmath,
    amsthm,
    amssymb,
    mathtools,
    mathrsfs,
    mathdots,
    stmaryrd,
    dsfont
}

\usepackage{tikz}
\usepackage{pgfplots}
\usepgfplotslibrary{polar}
\pgfplotsset{compat=newest}

\usepackage{enumitem}
\usepackage{multicol}

\usepackage{color}
\usepackage[svgnames]{xcolor}
\usepackage[dvipsnames]{xcolor}
\usepackage[most]{tcolorbox}

\usepackage{soul}

\usepackage{
    sansmath,
    textcomp,
    wasysym,
    xspace,
    setspace
}

\usepackage{fancyhdr}
\usepackage[normalem]{ulem}

\usepackage{thmtools}
\makeatletter
\renewcommand\thmt@autorefsetup{%
  \@xa\def\csname\thmt@envname autorefname\@xa\endcsname
    \@xa{\thmt@thmname}%
}
\makeatother

\usepackage{hyperref}
\usepackage{aliascnt}
\usepackage[capitalize]{cleveref}
\usepackage{lipsum}

\makeatletter
\let\@fnsymbol\@alph
\makeatother

\newcommand{\rout}{\bgroup\markoverwith{\textcolor{SuggestionColor}{\rule[0.5ex]{2pt}{0.4pt}}}\ULon}

\newtheorem{theorem}{Theorem}[section]
\newaliascnt{corollary}{theorem}
\newtheorem{corollary}[corollary]{Corollary}
\aliascntresetthe{corollary}
\newaliascnt{lemma}{theorem}
\newtheorem{lemma}[lemma]{Lemma}
\aliascntresetthe{lemma}
\newaliascnt{proposition}{theorem}
\newtheorem{proposition}[proposition]{Proposition}
\aliascntresetthe{proposition}
\newaliascnt{remark}{theorem}

\aliascntresetthe{remark}
\newaliascnt{question}{theorem}
\newtheorem{question}[question]{Question}
\aliascntresetthe{question}
\theoremstyle{definition}
\newaliascnt{definition}{theorem}
\newtheorem{definition}[definition]{Definition}
\aliascntresetthe{definition}
\newaliascnt{example}{theorem}
\newtheorem{example}[example]{Example}
\aliascntresetthe{example}
\newaliascnt{conjecture}{theorem}

\aliascntresetthe{conjecture}
\newaliascnt{fact}{theorem}

\aliascntresetthe{fact}
\newcommand{\N}{\mathbb{N}}

\newcommand{\R}{\mathbb{R}}
\newcommand{\C}{\mathbb{C}}

\newcommand{\F}{\mathbb{F}}

\newcommand{\A}{\mathcal{A}}

\DeclareMathOperator{\Out}{Out}
\DeclareMathOperator{\IBr}{IBr}
\newcommand{\An}[1]{\A_{#1}}
\newcommand{\Sn}[1]{\mathcal{S}_{#1}}

\newcommand{\GK}[1]{\overline{\Gamma}\left(#1\right)}

\definecolor{SuggestionColor}{RGB}{100,200,120}
\definecolor{DoNotTouchColor}{named}{teal}
\definecolor{TodoColor}{named}{red}

\title{The prime graph complements of $\A_{11}$- and $\A_{13}$-solvable groups}
\author{
    Logan Abounader \thanks{qhi15@txstate.edu; Department of Mathematics, Texas State University, 601 University Drive, San Marcos, TX 78666, USA.}
    \and 
    Austin Bell \thanks{acp202@txstate.edu; Department of Mathematics, Texas State University, 601 University Drive, San Marcos, TX 78666, USA.}
    \and 
    April Beun \thanks{abeun@gmu.edu; Department of Mathematical Sciences, George Mason University, 4400 University Drive Fairfax, Virginia 22030, USA.}
    \and 
    Elvis Cabrera \thanks{elviscal@mit.edu; Department of Mathematics, Massachusetts Institute of Technology, 77 Massachusetts Avenue Cambridge, MA 02139, USA}
    \and 
    Dakota Nobuyuki Goto \thanks{dgoto@fordham.edu; Department of Mathematics, Fordham University, 113 W 60th Street New York, NY 10023, USA.}
    \and 
    Thomas Michael Keller \thanks{keller@txstate.edu; Department of Mathematics, Texas State University, 601 University Drive, San Marcos, TX 78666, USA.}
}
\date{}

\begin{document}
\maketitle

\begin{abstract}
We study prime graph complements of $\A_n$-solvable groups, focusing on the cases $\A_{11}$ and $\A_{13}$. These provide the first classifications of $T$-solvable groups for which $|\pi(T)|>4$. We develop structural and representation-theoretic tools for analyzing the possible prime graph complements and use them to obtain complete characterizations for both $\A_{11}$-solvable and $\A_{13}$-solvable groups. The methods introduced also provide a framework for studying analogous classification problems for larger alternating groups.
\end{abstract}
 
\hfill\break

\hfill\break

\section*{Introduction}

For a finite group $G$, let $\pi(G)$ be the set of prime divisors of $\vert G \vert$. The \textit{prime graph} (or \textit{Gruenberg-Kegel graph}) $\Gamma(G)$ of $G$ is a simple graph with vertex set $\pi(G)$ and an edge $p-q \in \Gamma(G)$ if and only if there exists an element of order $pq$ in $G$. We focus our study on the complement of this graph, denoted $\GK{G}$. In \cite{GRUBER2015397}, the last author and his collaborators began the study of the prime graph for infinite families of groups by characterizing the prime graph complements of solvable groups. We state their result below as a theorem. 
\begin{theorem}[{\cite[Theorem 2.10]{GRUBER2015397}}]
\label{thm: solvable case}
    An unlabeled graph $\Xi$ is isomorphic to the complement of the prime graph of some solvable group if and only if $\Xi$ is 3-colorable and triangle-free.
\end{theorem}
One method for classifying the prime graphs of infinite families of groups has since focused on groups whose composition factors are either abelian or are isomorphic to some fixed non-abelian simple group $T$. We call such groups $T$-solvable. It is well known that $6$ divides the order of all finite nonabelian simple groups with the exception of the Suzuki groups. Additionally, by Burnside, we know that all nonabelian simple groups have at least three prime divisors in their order \cite{isaacs_finite_2008}. As such, the study of $T$-solvable groups begins with groups $T$ where $\pi({T})=\{2,3,p\}$ for a prime $p$. Collectively, \cite{florez2022primegraphs} and \cite{HuangKellerKissingerPlotnickRomaYang+2024+89+117} began and completed the classification for $\An{5}$-solvable groups. Although seemingly the simplest possible choice of $T$, $\An{5}$ proved difficult in ways unique to its small size. In $2024$, \cite{edwards2022primegraphsgroupsarithmetically} classified the prime graph complements for all $T$-solvable groups such that $\vert\pi(T)\vert=3$. Furthermore, their results included a characterization of the prime graph complements of all such groups whose composition factors have at most three distinct prime divisors in their order. 

The transition from $T$-solvable groups with $\vert\pi(T)\vert=3$ to ones with $\vert\pi(T)\vert=4$ required substantially new techniques. Rather than relying on arguments tailored to individual simple groups, \cite{keller2024classifyingprimegraphsfinite}, \cite{keller2024classificationprimegraphsoperatornamesz8}, \cite{keller2025criteriaclassifyingprimegraphsPSL2q} developed a general framework for studying prime graphs of $T$-solvable groups through their subgroup structure, Schur multipliers, and representation theory. Collectively, these results completed the classification of prime graph complements for ``most''
$T$-solvable groups for $\vert\pi(T)\vert \le 4$ (the exception being one family of groups), while simultaneously producing general tools applicable well beyond the specific classifications for which they were originally developed.

In this paper, we take a step beyond this by studying $\A_{11}$- and $\A_{13}$-solvable groups (where $\A_{n}$ denotes the alternating group of degree $n$);
see \cref{A_11 Characterization,A_13 Characterization}, respectively.
These are the first classifications for $T$-solvable groups in which $\vert\pi(T)\vert>4$, with $\pi(\A_{11})$ containing five primes and $\pi(\A_{13})$ containing six. We first establish structural results for $\A_n$-solvable groups, both graph theoretically and representation theoretically. We use the new machinery of \cite{amrutha2024cyclic} and \cite{prasad2025minimal} to entirely eliminate the need for complex character tables. This reduces what was quickly becoming an infeasible amount of computation in the results of \cref{Sec: Representations of An}. 

Central in our developments is \cref{lma: Brauer realization theorem}, which generalizes the realization techniques of \cite[Corollary 2.0.2]{alland2025realizinggruenbergkegelgraphstsolvable} by allowing the construction of a single extension whose elementary abelian normal subgroup decomposes into modules corresponding to different Brauer characters over different characteristics. This removes the need to realize all required edges within a single module and substantially broadens the class of graphs that can be constructed. For the cases of $\A_{11}$ and $\A_{13}$, these realization results combine with edge elimination arguments to yield complete characterizations of the prime graph complements containing triangles. Consequently, every realizable graph of this type arises from the module constructions developed here, eliminating the need to consider more complicated extensions.\\
At the end of the paper, we also mention some open questions that arose along the way.

\section*{Preliminary Notions}

We will use the following notational conventions throughout the paper. Let $G$ be a group, $\Xi$ be a graph, $V$ be a set of vertices, and $E$ be a set of edges. 
\begin{enumerate}

    \item We say $G$ is \textit{strictly} $T$-solvable for a nonabelian simple group $T$ when $G$ is $T$-solvable and there exists a composition factor of $G$ isomorphic to $T$.

    \item We will use atlas notation to describe extensions and normal series of finite groups. In particular, if $G$ has a normal series with factors $X_1,\ldots,X_k$, we may write $G\cong X_1.X_2.\dots.X_k$.
    
    \item We write $m.G$ to denote a perfect central extension of $G$ by a cyclic group of order $m$. In particular, $2.\A_n$ denotes the unique Schur double cover of $\A_n$.
    
    \item We write $\GK{G}$ to denote the graph theoretic complement of the prime graph of $G$.
    
    \item If $\Xi$ is a graph and $p,q\in\Xi$, then we write $p-q\in\Xi$ to indicate that $\{p,q\}$ is an edge of $\Xi$.
    
    \item We write $\Xi[V]$ to denote the induced subgraph of $\Xi$ on the vertices in $V$.
    
    \item We write $\Xi\setminus V$ to denote the induced subgraph of $\Xi$ obtained by deleting the vertices in $V$. Respectively, we write $\Xi\setminus E$ to denote the subgraph of $\Xi$ obtained by deleting the edges in $E$.
    
    \item If $\mathcal P$ is a set of vertices of a graph $\Xi$, then $N_{\Xi}(\mathcal P)$ denotes the set of vertices in $\Xi\setminus\mathcal P$ adjacent to at least one vertex in $\mathcal P$. When the graph is clear from context, we write $N(\mathcal P)$.
    
    \item We write $\IBr_p(G)$ to denote the set of irreducible $p$-Brauer characters of $G$. Thus, if $\chi\in\IBr_p(G)$, then $\chi$ is an irreducible Brauer character of $G$ in characteristic $p$.

\end{enumerate}

\noindent It is easily seen that $\GK{G}[\pi(H)]$ can at most contain the edges in $\GK{H}$ for any section $H$ of $G$. Specifically, $\GK{G}[\pi(H)] \subseteq \GK{H}$; this is a fact that we will employ readily and without reference. 

\medskip
From \cite{keller2024classifyingprimegraphsfinite} we get the Frobenius criterion, which characterizes the subgroups that can act as a Frobenius complement.
\begin{definition}[Frobenius Criterion \cite{keller2024classifyingprimegraphsfinite}]
\label{Frobenius Criterion}
    A $p$-group is said to satisfy the \textit{Frobenius Criterion} if it is cyclic, dihedral (including the Klein-4 group), or generalized quaternion.
\end{definition}
\noindent Every subgroup of any such group also satisfies the Frobenius criterion. Notably, since $\A_n \le \A_m$ for $n \le m$, we know that if a Sylow $p$-subgroup of $\A_n$ does not satisfy the Frobenius criterion, neither will the Sylow $p$-subgroups of $\A_m$ for all $m>n$. We will repeatedly use the following consequence of the Frobenius criterion to rule out edges in the prime graph complement.
 
\begin{proposition}[{\cite[Lemma 2.1.2 and Proposition 2.2.2]{keller2024classifyingprimegraphsfinite}}]
\label{prop: frobenius criterion consequence}
    
Let $G$ be a $T$-solvable group and $\pi(T)=\pi(\text{Aut}(T))$. Assume the Sylow $r$-subgroups of $T$ do not satisfy the Frobenius criterion for some
$r \in \pi(T)$. For all $p \in \pi(G) \setminus \pi(T)$, we have
$r-p \notin \GK G$.
\end{proposition}

\noindent The Frobenius criterion will allow us to determine when fixed points can or cannot occur in the actions we consider. Furthermore, whenever we reference fixed points, we always mean \textit{nontrivial} fixed points.

\medskip
Classifications of prime graph complements for $T$-solvable groups have consistently treated edges among primes in $\pi(T)$ and edges connecting $\pi(T)$ to outside primes as distinct problems. For instance, \cite{GRUBER2015397,florez2022primegraphs,HuangKellerKissingerPlotnickRomaYang+2024+89+117} separately analyze the possible triangles on $\{2,3,5\}$ and the possible edges between $\{2,3,5\}$ and an arbitrary prime $r$. Similarly, \cite{edwards2022primegraphsgroupsarithmetically,keller2024classificationprimegraphsoperatornamesz8,keller2025criteriaclassifyingprimegraphsPSL2q} follow the same pattern for their respective choices of $T$. We formalize this distinction as follows.
\begin{definition}
    Let $G$ be a $T$-solvable group and $p-q \in \GK{G}$. We say $p-q$ is an \textit{internal} edge if $p,q \in \pi(T).$
    We say $p-q$ is a \textit{bridging} edge if $p \in \pi(T)$ and $q \not\in \pi(T).$ 
\end{definition}

The following result guarantees the existence of a $T$-solvable group with a specified prime graph complement when the internal edges and bridging edges are well-understood.
\begin{lemma}\cite[Lemma 0.13]{keller2025criteriaclassifyingprimegraphsPSL2q}
\label{lma: psl2q realization}
    Let $T$ be a nonabelian simple group, $E$ a strictly $T$-solvable group, $\Xi$ a graph, and $X$ a set of $|\pi(E)|$ vertices of $\Xi$. Suppose $\Xi \setminus X$ is triangle-free and $\Xi$ has a $3$-coloring such that vertices in $N(X) \setminus X$ are monochromatic. Further suppose that:
    \begin{enumerate}
        \item There exists a graph isomorphism $\phi$ from the subgraph of $\Xi[X]$ to $\GK{E}$.
    
        \item Given $v \in \Xi \setminus X$, there exists a complex irreducible representation $\rho$ of $E$ such that for all $x \in X$, $x-v \in \Xi$ if and only if elements of order $\phi(x)$ of $E$ act without fixed points in $\rho$.
    \end{enumerate}
    Then there exists a group $G$ such that $\Xi \cong \GK{G}$, and $G$ is $T$-solvable.
\end{lemma}

\cref{lma: psl2q realization} shows that once we have a group $E$ that realizes the internal edges we want, and representations with agreeable properties, every bridging edge from $\pi(E)$ to an outside prime can be attached via a module achieving the required Frobenius action. What the lemma does not supply is such a group $E$ in the first place, nor a way to assure there exists a representation which acts as intended. For $\An{n}$, we handle the latter in \cref{Sec: Representations of An}. As for the former, for a given nonabelian simple group $T$, we still need a way to realize an arbitrary prescribed pattern of internal edges on $\pi(T)$ without introducing any new prime vertices. We do so through the following result from \cref{Sec: Edges}:

\medskip\noindent\textbf{\cref{lma: Brauer realization theorem}.}
{\itshape
    Let $T$ be a nonabelian finite simple group. Then, for any collection of primes $R\subseteq\pi(T)$ and any choice of nonempty sets of irreducible Brauer characters of $T$ in the corresponding characteristics, there exists a $T$-solvable group $G$ whose prime graph complement is obtained from $\GK{T}$ by removing all edges between distinct primes in $R$, together with precisely those edges between $r\in R$ and $q\in\pi(T)$ for which some chosen Brauer character in characteristic $r$ has a nontrivial fixed point under an element of order $q$.
}

\medskip
By selecting irreducible Brauer characters at each prime $r \in R\subseteq\pi(T)$, it constructs an extension $G=V \rtimes T$ whose prime graph complement is $\GK{T}$ with a chosen set of internal edges deleted, where $V$ is assembled entirely from modules of characteristics already occurring in $\pi(T)$. Together, the two results let us realize internal and bridging edges in two independent steps.

\section{On the Prime Graph Complement of $\A_n$}
\label{Sec: Prime Graph Complement}

It is well known that if $G$ is a $T$-solvable group, then the internal edges of
$\GK{G}$ are contained in the edges of $\GK{T}$. Thus, an important first step towards our characterization is to understand the prime graph complement of $\A_n$. The following lemma gives necessary and sufficient conditions for the existence of edges in this graph.

\begin{lemma}\cite[Proposition 1.1]{VasilievVdovin2005}
\label{lma: A_n graph characterization}
    Let $n \geq 5$, and let $p,q \in \pi(\A_n)$ be distinct primes with $p>q$. Then $p-q \in \overline{\Gamma}(\A_n)$ if and only if one of the following conditions holds:
    \begin{enumerate}
        \item $p$ and $q$ are odd primes and $p+q > n$;
        \item $q=2$ and $p+4 > n$.
    \end{enumerate}
\end{lemma}

\noindent In the following result, the symbol $\chi$ denotes the chromatic number of a graph. Throughout the remainder of the paper, $\chi$ denotes the character of a representation.

\begin{corollary}
\label{chromatic number}
    Let $n\ge 9$, $p_i$ be the smallest prime such that $p_i+p_{i+1}>n$, and $P_>(n):=\{p \text{ prime} \mid p_i \le p \le n\}$. 
    Then, $\chi(\GK{\A_n})=\vert P_>(n) \vert$. In particular, for all $k\in\N$, there exists $n\in \N$ such that $\chi(\GK{\A_n})\geq k$. 
\end{corollary}
\begin{proof}
    Define $w:\pi(\A_n)\to \R$ by $w(2)=4$ and $w(p)=p$ for every odd prime $p\in\pi(\A_n)$. Then, for distinct primes $p,q\in\pi(\A_n)$, we have $p-q\in \GK{\A_n}$ if and only if $w(p)+w(q)>n$. Hence $\GK{\A_n}$ is a threshold graph, with weight function $w$ and threshold $n$. 
    
    Since adjacency in $\GK{\A_n}$ is determined by the condition $w(p)+w(q)>n$, vertices with larger weights have at least as many neighbors as vertices with smaller weights. Hence a clique of maximum size is obtained by taking the largest possible set of vertices whose two smallest weights still sum to more than $n$; for $n\ge 9$, this set is $P_>(n)$. Thus, the largest clique is $\GK{\A_n}[P_>(n)]$. And, since threshold graphs are perfect, their chromatic number equals their clique number \cite{MahadevPeled}. 
    
    For $x\in\N$, define $P(x)$ to be the prime counting function. By the Prime Number Theorem, $P(n) \sim \frac{n}{\log{n}}$. So, $P(n)-P(n/2) \sim \frac{n}{\log{n}}-\frac{n/2}{\log{n/2}} \sim \frac{n}{2\log{n}}$. Since $\vert P_>(n)\vert \geq P(n)-P(n/2)$, it follows that $\vert P_>(n)\vert$ grows without bound. Thus, we have that for any $k \in \N$ there exists $n$ such that $\vert P_>(n) \vert \geq k$, which implies $\chi(\GK{\A_n}) \geq k$.
\end{proof}

To illustrate the above, consider the prime graph complement of $\An{17}$: 
\begin{center}
    \begin{tikzpicture}[
        vertex/.style={circle,draw,fill=white,minimum size=8mm},
        every edge/.style={thick}
    ]
        
        \node[vertex] (5)  at (3, 3) {$5$};
        \node[vertex] (2)  at (-3, 3) {$2$};
        \node[vertex] (3)  at (-2.5, 1.5) {$3$};
        
        \node[vertex, fill=red!40, line width=0.8pt, draw=white] (7)  at (-1.5, 0.40192379) {$7$};
        \node[vertex, fill=red!40, line width=0.8pt, draw=white] (11) at (0, 1.28) {$11$};
        \node[vertex, fill=red!40, line width=0.8pt, draw=white] (13) at (1.5, 0.40192379) {$13$};
        \node[vertex, fill=red!40, line width=0.8pt, draw=white] (17) at (0, 3) {$17$};
        
        \draw (17)--(2);
        \draw (17)--(3);
        \draw (17)--(5);
        \draw (5)--(13);

        \draw[red, line width=2pt] (17)--(7);
        \draw[red, line width=2pt] (17)--(11);
        \draw[red, line width=2pt] (17)--(13);
        \draw[red, line width=2pt] (7)--(11);
        \draw[red, line width=2pt] (7)--(13);
        \draw[red, line width=2pt] (11)--(13);
    
    \end{tikzpicture}
\end{center}

\section{Representations of $\A_n$}
\label{Sec: Representations of An}

Note that since \cite{edwards2022primegraphsgroupsarithmetically} and \cite{keller2024classifyingprimegraphsfinite} have completed the classification of prime graph complements of $\A_n$-solvable groups for $n \leq 10$, we simplify many of our general results by implementing lower bounds on $n$. 

\medskip

Understanding which internal and bridging edges can occur in the prime graph complements of $\A_n$-solvable groups is required in our approach to characterizing these graphs. For the prime graph complements of a $T$-solvable group, \cref{lma: psl2q realization} and \cite[Corollary 2.0.2]{alland2025realizinggruenbergkegelgraphstsolvable} relate the existence of these edges to the complex and Brauer representations of certain extensions of T. For $T\in\{\An{11}, \An{12}, \An{13}, \An{14}, \An{15}, \An{16}\}$, in the case of internal edges, \cite[Theorem A]{alland2025realizinggruenbergkegelgraphstsolvable} restricts the set of relevant extensions to Schur covers and extensions above by subgroups of $\Out(T)$. (This restriction is conjectured to hold for all $\A_n$.)

\medskip
\noindent
Specifically, we want to understand when elements of order $p$ in $\Sn{n}$, $\A_n$, and $2.\A_n$ act without fixed points in a representation. A recently published result of Amrutha P., Amritanshu Prasad, and Velmurugan S. completely answers this question for representations of $\A_n$ over $\C$:

\begin{theorem}[{\cite[Theorem C]{amrutha2024cyclic}}]
\label{thm: A_n complex rep}
    Let $n>8$ and $p$ be a prime with $p\leq n$. There exists an irreducible complex representation of $\A_n$ such that every element of order $p$ acts without fixed points if and only if $n=p$.
\end{theorem}

We extend the above result from \cite{amrutha2024cyclic} to $2.\A_n$. In order to do so, we require a number of lemmas. We will begin by demonstrating that for a large class of odd primes $p\in \pi(\A_n)$, there will always be at least two elements, $g_1\in \A_n$ and $g_2\in 2.\A_n$, such that $g_1$ and $g_2$ have order $p$ and act with fixed points.

\begin{lemma}
\label{lem:CpCp-no-fpf}
    Let $p$ be a prime and let $F$ be a field with $\operatorname{char}(F)\neq p$. Then for every nontrivial $F[C_p\times C_p]$-module $V$, there exists a nontrivial element of $C_p \times C_p$ that acts with a fixed-point on the module.
\end{lemma}
\begin{proof}
    This follows from the fact that $C_p \times C_p$ can never act as the Frobenius complement of a Frobenius group \cite[Theorem 5.3.1]{wolf2011spaces}.
\end{proof}

\begin{corollary}
\label{cor:Cr-leq-n/2-obstruction}
    Let $n \geq 8$, $p \leq n/2$ be a prime, and $F$ be a field with $\operatorname{char}(F) \not= p$. Then for every irreducible representation of $\A_n$, $2.\A_n$, or $\Sn{n}$ over $F$, there exists an element of order $p$ that acts with a fixed point in the representation.
\end{corollary}
\begin{proof}
    From the Sylow $p$-subgroups of $\Sn{n}, \A_n \text{, and } 2.\A_n$, we have subgroups in each that are isomorphic to $C_p\times C_p$. Then the restriction of any nonzero irreducible module $V$
    $$
    V\!\downarrow_{C_p\times C_p}
    $$
    is a nonzero $F[C_p\times C_p]$-module. By \cref{lem:CpCp-no-fpf}, there exists an element of order $p$ acting on $V$ with a fixed point.
\end{proof}

\noindent
We use \cref{cor:Cr-leq-n/2-obstruction} for elements of order less than or equal to five in the complex representations of $2.\A_n$, but it also handles all elements of order $p \leq n/2$ in the modular representations of $\A_n$.

\begin{lemma}
\label{lma: 2.A_n complex spin reps}
    Let $n \geq 10$ and $p$ be an odd prime such that $5 < p \le n$. Then every irreducible complex spin representation of $2.\A_n$ contains an element of order $p$ that acts with fixed points.
\end{lemma}
\begin{proof}
    Let $\rho$ be an arbitrary spin representation of $2.\A_n$. It is sufficient to show that $1$ is an eigenvalue of some element of order $p$ in $2.\A_n$. Let $g'\in \A_n$ be a $p$-cycle. Taking a preimage of $g'$ in $2.\A_n$ via the canonical homomorphism, we obtain an element $g$. Since the cycle type of $g'$ does not contain a 3 or a 5, we avoid the four families of exceptions in \cite[Theorem 1.2]{prasad2025minimal}. Hence, the minimal polynomial of $\rho(g)$ is
    $$x^p\pm 1.$$
    
    \medskip
    First, we consider the case where
    $$p \equiv 1 \text{ or } 7 \pmod 8.$$
    By \cite[Corollary 4.3]{prasad2025minimal}, we have
    $$\rho(g)^p=1.$$
    Consequently, the order of $g$ must be $p$ and its minimal polynomial divides $x^p-1.$ Hence, the minimal polynomial from \cite[Theorem 1.2]{prasad2025minimal} is exactly
    $$x^p-1,$$
    of which $1$ is a root.
    
    \medskip
    Then, in the case
    $$p \equiv 3 \text{ or } 5 \pmod 8,$$
    by \cite[Corollary 4.3]{prasad2025minimal}, we have $\rho(g)^p=-1.$ It follows that $g$ has order $2p$, and consequently $g^2$ is an element of order $p$ with $\rho(g^2)^p=1$. $g^2$ then again has minimal polynomial
    $$x^p-1,$$
    so again we have $1$ as a root.
\end{proof}

\noindent
We now have sufficient tools to prove the following theorem.

\begin{theorem}
\label{thm: 2.A_n complex rep}
    Let $n \geq 10$ and prime p with $p\leq n$. There exists an irreducible complex representation of $2.\A_n$ such that every element of order $p$ acts without fixed points if and only if $n=p$.
\end{theorem}
\begin{proof}
    We find that there exists an element of $2.\A_n$ of order $p$ that has a fixed point in any irreducible complex spin representation of $2.\A_n$ for $p \le 5$ from \cref{cor:Cr-leq-n/2-obstruction}, and for $p>5$ from \cref{lma: 2.A_n complex spin reps}. Therefore, the only case where there exists an irreducible representation such that all elements of $2.\A_n$ of order $p$ act without fixed points for $n \geq 10$ is $n=p$, arising from \cref{thm: A_n complex rep}.
\end{proof}

\noindent The following example is computed exactly in James' book \cite{james1978} as a representation of $\Sn{n}$. Since the partition $\lambda = (n-1,1) \vdash n$ is not self-conjugate for $n > 3$, the restriction to $\A_n$ does not split by Clifford's Theorem \cite{Cliff}. Hence, the following two examples of $\F_q[\Sn{n}]$-modules that allow elements of $\Sn{n}$ of order $p$ to act without fixed points also translate directly to $\F_q[\A_n]$-modules that allow elements of $\A_n$ of order $p$ to act without fixed points. Since every representation of $\A_n$ is also a representation of $2.\A_n$, these in turn translate directly to $\F_q[2.\A_n]$-modules that allow elements of $2.\A_n$ of order $p$ to act without fixed points. To simplify our notation, we will define a group $H_n \in \{\Sn{n},\A_n,2.\A_n\}$ and consider the following examples as $\F_q[H_n]$-modules.

\begin{example}[Example~5.1 from \cite{james1978}]
\label{ex:augment-modular}
    Let $n > 3$, let $\F_q$ be a field of characteristic $q \leq n$, and let
    $$
    M=\F_q^n
    $$
    be the permutation module of $\Sn{n}$ with basis $e_1,\dots,e_n$. Define $W$, the augmentation submodule of the permutation module, as the kernel of the augmentation map:
    $$
    W:=\left\{
    (x_1,\dots,x_n)\in \F_q^n :
    \sum_{i=1}^n x_i=0
    \right\},
    $$
    and let $\mathbf{1}=(1, \dots, 1)$ with
    $$
    U=\langle \mathbf{1} \rangle.
    $$
    Then:
    \begin{enumerate}
        \item If $q\nmid n$, then $W$ is an irreducible $\F_q[H_n]$-module of dimension $n-1$.
        \item If $q\mid n$, then $U$ is the unique proper nonzero $\F_q[H_n]$-submodule of $W$, and $W/U$ is an irreducible $\F_q[H_n]$-module of dimension $n-2$.
    \end{enumerate}
\end{example}

\begin{corollary}
\label{cor:Ap}
    Let $p>3$ and $q$ be distinct primes. Let $W$ be the augmentation $\F_q[H_p]$-module, and let $g\in H_p$ be a $p$-cycle. Then $g$ acts without fixed points on $W$.
\end{corollary}
\begin{proof}
    Since $q \nmid p$, \cref{ex:augment-modular} shows that $W$ is an irreducible $\F_q[H_p]$-module of dimension $p-1$. The fixed points of $g$ on the permutation module $\F_q^p$ are exactly the vectors of the form
    $$
    (a,\dots,a).
    $$
    However, such a vector lies in $W$ only if
    $$
    pa=0.
    $$
    Since $q\neq p$, this implies $a=0$. Hence the fixed point space of $g$ in $W$ is spanned by the zero vector, so $g$ acts without nontrivial fixed points on $W$.
\end{proof}

\begin{corollary}
\label{cor:Ap+1_q|p}
    Let $p$ be an odd prime and let $q$ be a prime dividing $p+1$. Let
    $$Q=W/U,$$
    where $W$ is the augmentation $\F_q[H_{p+1}]$-module and $U=\langle \mathbf{1} \rangle$. Then:
    
    \begin{enumerate}
    \item $Q$ is an irreducible $\F_q[H_{p+1}]$-module of dimension
    $p-1$.
    \item Every element of order $p$ acts without fixed points on $Q$.
    \end{enumerate}
\end{corollary}
\begin{proof}
    Since $q\mid(p+1)$, \cref{ex:augment-modular} shows that $Q=W/U$ is an irreducible module of dimension $p-1$. Let $g=(1\,2\,\cdots\,p)$. A vector
    $$
    v=(v_1,\ldots,v_{p+1})\in W
    $$
    is fixed by $g$ if and only if
    $$
    v_1=\cdots=v_p=a,\qquad v_{p+1}=b
    $$
    for some $a,b\in\F_q$. Since $v\in W$,
    $$
    pa+b=0.
    $$
    As $q\mid(p+1)$, $p\equiv-1\pmod q$, so $b=-pa=a$. Hence the fixed point space of $g$ in $W$ is exactly $U$, and since $p \not= q$, this implies that the fixed point space of $g$ in $Q$ is spanned by the zero vector, so $g$ acts without nontrivial fixed points on $Q$.
\end{proof}

\section{On Bridging and Internal Edges}
\label{Sec: Edges}

Recall that if $G$ is $T$-solvable, for some nonabelian simple group $T$, we say that $p-q$ is a bridging edge in $\GK{G}$ if and only if $p \in \pi(T)$ and $q\in \pi(G)\setminus \pi(T)$. We begin this section by deriving necessary conditions for the existence of bridging edges in the case that $T=\A_n$.

\begin{lemma}\label{bridging edges char}
    Let $n>10$, and let $G$ be a strictly $\A_n$-solvable group. If there exist primes $r\in\pi(\A_n)$ and $p\in\pi(G)\setminus\pi(\A_n)$ such that $p-r\in\GK{G}$, then $n$ is prime and $r=n$.
\end{lemma}
\begin{proof}
   Suppose first that $r\in\{2,3,5\}$. The Sylow $r$-subgroups of $\An{11}$ are, respectively, isomorphic to $C_2\wr C_2\wr C_2$, $C_3\wr C_3$, and $C_5\times C_5$ \cite{GAP}, which do not satisfy the Frobenius criterion. Since $\An{11}\leq \A_n$ for $n\geq 11$, the Sylow $r$-subgroups of $\A_n$ do not satisfy the Frobenius criterion either. Hence, \cref{prop: frobenius criterion consequence} implies that $p-r\notin\GK{G}$. Thus, $r>5$. If $r\neq n$, then \cref{thm: A_n complex rep,thm: 2.A_n complex rep}, together with \cite[Corollary 2.2.6] {keller2024classifyingprimegraphsfinite}, implies that $p-r\notin\GK{G}.$ Therefore, $r=n$.
\end{proof}

We now turn to the study of internal edges in $T$-solvable groups, i.e., edges incident only to primes dividing the order of $T$. A cornerstone of this approach is \cite[Corollary 2.0.2]{alland2025realizinggruenbergkegelgraphstsolvable}, which shows that the edges lost through actions on elementary abelian groups can be described in terms of Brauer characters. The following lemma extends this idea by considering actions on modules of mixed characteristic. This allows us to realize the deletion of a broader class of edges from the prime graph complement of $T$.

\begin{lemma}
\label{A13 backwards direction}
    Let $G$ be a finite group, let $R\subseteq\pi(G)$, and for each $r\in R$, let $V_r$ be a nontrivial elementary abelian $r$-group on which $G$ acts. Set $B:=\prod_{r\in R}V_r$. Then, for each $r\in R$, there exists a nonempty subset $Y_r\subseteq\IBr_r(G)$ such that, upon defining
    $$
    A_r:=\bigcup_{\chi\in Y_r}\left\{r-q\in\GK{G}\mid \exists g\in G,\ |g|=q,\ 
    \frac{1}{|g|}\sum_{x\in\langle g\rangle}\chi(x)>0\right\},
    $$
    we have
    \begin{equation*}
    \GK{B\rtimes G}\cong \GK{G}\setminus\left(\bigcup_{r\in R}A_r\cup\{r-s\mid r,s\in R,\ r\neq s\}\right).
    \end{equation*}
\end{lemma}
\begin{proof}
    For each $r\in R$, \cite[Corollary 2.0.2]{alland2025realizinggruenbergkegelgraphstsolvable} gives a nonempty subset $Y_r\subseteq\IBr_r(G)$ such that
    $$
    \GK{V_r\rtimes G}\cong\GK{G}\setminus A_r.
    $$
    Applying \cite[Lemma 3.1]{keller2024classificationprimegraphsoperatornamesz8} repeatedly to the groups $V_r$ gives
    \begin{equation*}
    \GK{B\rtimes G}\cong
    \left(\bigcap_{r\in R}\GK{V_r\rtimes G}\right)
    \setminus
    \{r-s\mid r,s\in R,\ r\neq s\}.
    \end{equation*}
    Therefore,
    \begin{equation*}
    \GK{B\rtimes G}\cong \GK{G}\setminus\left(\bigcup_{r\in R}A_r\cup\{r-s\mid r,s\in R,\ r\neq s\}\right).
    \end{equation*}
\end{proof}

Next, we show that, once the edge sets $A_r$ are defined from the chosen Brauer character data, removing those edges from $\GK{T}$ can be realized by passing to a semidirect product $V\rtimes T$, where $V$ is a direct product of elementary abelian $r$-groups.

\begin{lemma}
\label{lma: Brauer realization theorem}
    Let $T$ be a nonabelian finite simple group, let $R\subseteq\pi(T)$. For each
    $r\in R$, let $\emptyset \neq Y_r\subseteq\IBr_r(T)$.
    For each $r\in R$ and $\chi\in Y_r$, define
    $A_r:=\bigcup_{\chi \in Y_r}\{r-q\in\GK{T}\mid \exists g\in T,\ |g|=q,\ 
    \frac{1}{|g|}\sum_{x\in\langle g\rangle}\chi(x)>0\}$.Then, there exists a group $V$ such that $G:=V\rtimes T$ is $T$-solvable and
    
    $$\GK{G}\cong\GK{T}\setminus\left(\bigcup_{r\in R}A_r\cup\{r-s\mid r,s\in R,\ r\neq s\}\right).$$
\end{lemma}
\begin{proof} 
    By the construction in the proof of \cite[Theorem 3.7]{keller2024classificationprimegraphsoperatornamesz8}, for each $r\in R$, there exists a nontrivial elementary abelian $r$-group $V_r$ and a homomorphism $\varphi_r:T\to\operatorname{Aut}(V_r)$ such that $\GK{V_r\rtimes_{\varphi_r}T}\cong \GK{T}\setminus A_r$. Set $V:=\bigoplus_{r\in R}V_r$, and let $T$ act componentwise on $V$; that is, define $\varphi:T\to\operatorname{Aut}(V)$ by $\varphi = \bigoplus_{r\in R} \varphi_r$. Now, by applying \cite[Lemma 3.1]{keller2024classificationprimegraphsoperatornamesz8} inductively to the groups $V_r$, we obtain our desired result. Therefore, $\GK{G}\cong \left(\bigcap_{r\in R}\GK{V_r\rtimes_{\varphi_r}T}\right) \setminus\{r-s\mid  r,s\in R,\ r\neq s\}$. By the definition of each $V_r$, we have
    $$\bigcap_{r\in R}\GK{V_r\rtimes_{\varphi_r}T}=\GK{T}\setminus\bigcup_{r\in R}A_r.$$ 
    Consequently, 
    $$\GK{G}\cong\GK{T}\setminus\left(\bigcup_{r\in R}A_r\cup\{r-s\mid r,s\in R,\ r\neq s\}\right).$$
\end{proof}

\section{Characterization of $\A_n$-solvable Groups} \label{Sec: A11 and A13}

Using the tools developed so far, we now characterize the graphs realizable as prime graph complements of $\A_n$-solvable groups for $n\in \{11,13\}$.

\subsection{$\An{11}$-solvable Groups}

First, we use \cref{lma: A_n graph characterization} to illustrate the prime graph complement of $\An{11}$.

\begin{center}
    \begin{tikzpicture}[
        vertex/.style={circle,draw,fill=white,minimum size=8mm},
        every edge/.style={thick}
    ]
    
    \node[vertex] (11) at (0,2) {$11$};
    
    \node[vertex] (2) at (-4,0) {$2$};
    \node[vertex] (3) at (-1.4,0) {$3$};
    
    \node[vertex] (5) at (1.4,0) {$5$};
    \node[vertex] (7) at (4,0) {$7$};
    
    \draw (11)--(2);
    \draw (11)--(3);
    \draw (11)--(5);
    \draw (11)--(7);
    
    \draw (5)--(7);
    
    \end{tikzpicture}
\end{center}

In order to characterize the prime graph complements of $\An{11}$-solvable groups, we must understand the restrictions of bridging edges, possible triangles, and colorability in this setting. These are the same structural features that have driven related classifications, so they provide a natural starting point. In the following lemma, we focus on strictly $\An{11}$-solvable groups, since the solvable case was already characterized in \cite{GRUBER2015397}. 

\begin{lemma} 
\label{thm:fun stuff}
    Let $G$ be a strictly $\An{11}$-solvable group. Then $\overline \Gamma (G)$ satisfies all of the following:
    \begin{enumerate}
        \item There are no edges $r - p$ for $r\in \pi(\An{11})\setminus \{11\}$ and $p\in \pi(G) \setminus\pi(\An{11}).$
        \item The graph $\overline \Gamma (G)$ is 3-colorable and  has a coloring such that $N(\pi(\An{11})) \setminus\pi(\An{11})$ is of one color.
        \item If $\overline \Gamma (G)$ has a triangle, then its vertices are exactly $\{5,7,11\}.$
    \end{enumerate}
\end{lemma}
\begin{proof}
   First, observe that (1) follows directly from \cref{bridging edges char}. Also, from \cite[Lemma 2.3.5]{keller2024classifyingprimegraphsfinite}, $\overline \Gamma (G)$ is 3-colorable and has a 3-coloring where all neighbors of $11$ not in $\pi(\An{11})$ are the same color; hence (2) is satisfied. Furthermore, from \cite[Lemma 2.3.5]{keller2024classifyingprimegraphsfinite} it must be that every triangle in $\GK{G}$ is contained in $\overline \Gamma (G)[\pi(\An{11})].$ Since $\overline \Gamma (\An{11})$ contains exactly one triangle with vertices $\{5,7,11\},$ (3) is satisfied.
\end{proof}

We have now developed the tools needed to characterize the graphs realizable as prime graph complements of $\An{11}$-solvable groups. We next prove this characterization.
 
\begin{theorem}
\label{A_11 Characterization}
    A graph $\Xi$ is isomorphic to the prime graph complement $\GK{G}$ of some $\An{11}$-solvable group $G,$ if and only if $\Xi$ is 3-colorable and one of the following is satisfied: 
    \begin{enumerate}
        \item $\Xi$ is triangle-free. 
        \item $\Xi$ contains a subset of vertices $X=\{a,b,c,d,e\}$ where $\Xi\setminus X$ is triangle-free, $\Xi$ has a coloring such that $N(X)\setminus X$ is monochromatic, and the following are true:
        \begin{enumerate}
            \item There is exactly one triangle in $\Xi$ with vertices $\{a,b,c\}.$ 
            \item If $p-r\in \Xi$ with $p\in X$ and $r\in \Xi \setminus X$, then $p=a.$ 
            \item If $p-r\in \Xi$ with $r\in \{d,e\}$ then $p=a.$
        \end{enumerate}
    \end{enumerate}
\end{theorem}
\begin{proof}
    We will first prove the forward direction. Let $G$ be an $\An{11}$-solvable group. If $G$ is solvable, then $\GK{G}$ is 3-colorable and $(1)$ is satisfied \cite{keller2024classifyingprimegraphsfinite}. Otherwise, $G$ is strictly $\An{11}$-solvable. First, note that by \cref{thm:fun stuff}, the graph $\GK{G}$ is 3-colorable  such that $N(\pi(\An{11}))\setminus\pi(\An{11})$ is monochromatic. If $\GK{G}$ does not contain a triangle, then $\GK{G}$ satisfies (1). If $G$ does contain a triangle, then again by \cref{thm:fun stuff}, the vertices are exactly $\{11,7,5\}$ and there are no edges $r - p$ for $r\in \pi(\An{11})\setminus\{11\}$ and $p\in \pi(G)\setminus\pi(\An{11}).$ Any labeling consistent with $a=11,\{b,c\}=\{5,7\}, \{d,e\}=\{2,3\}$ satisfies 2(a) and 2(b). Also, the vertices $2$ and $3$ cannot have neighbors other than $11$ in $\GK{\An{11}}$, so 2(c) is satisfied. 

    We now turn to the backwards direction. Suppose we have some 3-colorable graph $\Xi$ that satisfies (1) or (2). If $\Xi$ satisfies $(1)$, then there exists a solvable group $G$ with prime graph complement isomorphic to $\Xi$ by \cite{GRUBER2015397}. Otherwise, $\Xi$ satisfies (2). Observe that we can always realize $\Xi[X]$ as the prime graph complement of some $\A_{11}$-solvable group $F$. We have three cases,
    \begin{enumerate}
        \item If both $d$ and $e$ are connected to $a,$ then choose $F=\An{11}$.
        \item If both $d$ and $e$ are isolated vertices, then choose $F=C_2\times C_3 \times \An{11}$.
        \item If either $d$ or $e$ is isolated, then choose $F=C_2\times \An{11}$.
        \end{enumerate}
    
    We now show that for each vertex $r \in X \setminus S$, there exists a complex representation $\rho$ of $F$ such that for all $p \in S$, $r-p \in \Xi$ if and only if all elements of order the prime associated to $p$ act without fixed points in $\rho$. In the case where $N(r) \cap S = \emptyset$, the trivial representation suffices. Consider then, the case where $r$ forms a bridging edge with some $p$. We know from 2(b) that if $p-r \in \Xi$, then $p=a$, which we associated to $11$. By \cref{thm: A_n complex rep}, there exists a representation $\tilde{\rho}$ of $\An{11}$ in which only elements of order $11$ act without fixed points.  Let $\pi: F \to \An{11}$ be the natural projection map. We can then define a representation of $F$, $\rho:=\tilde{\rho} \circ \pi$. Notice that $F = N \times \An{11}$, where $N$ is one of $C_2$, $C_3$, or $C_2\times C_3$. Since, by construction, $ \ker\rho = N$ and $11 \nmid |N|$, one can see that $|g|=11$ implies $|\rho(g)|=11$ for all $g \in F$. Consequently, $\rho$ is the desired representation of $F$. 

    Finally, we apply \cref{lma: psl2q realization} to extend the realization of $\Xi[X]$ to the entire graph $\Xi$ and conclude there exists an $\An{11}$-solvable group $G$ such that $\GK{G} \cong \Xi$.
   
\end{proof}

\subsection{$\An{13}$-solvable Groups}

In this section, we find a characterization for prime graphs of $\An{13}$-solvable groups. We begin by drawing the prime graph complement of $\An{13}$ below.
    
\par\medskip
    
{\centering
    \begin{tikzpicture}[
        scale=0.85,
        every node/.style={
            circle,
            draw,
            minimum size=8mm,
            inner sep=0pt,
            font=\large
        }
    ]
    \node (v2)  at (0,2.5)   {$2$};
    \node (v3)  at (0,0.5)   {$3$};
    \node (v5)  at (6,2.5)   {$5$};
    \node (v7)  at (6,0.5)   {$7$};
    
    \node[fill=red!40, line width=0.8pt, draw=white] (v11) at (3,2.5) {$11$};
    \node[fill=red!40, line width=0.8pt, draw=white] (v13) at (3,0.5) {$13$};

    \draw (v3) -- (v11);
    \draw (v7) -- (v11);
    \draw (v3) -- (v13);
    \draw (v7) -- (v13);
    
    \draw (v2) -- (v11);
    \draw (v5) -- (v11);
    \draw[white, line width=3.5pt] (v2) -- (v13);
    \draw[white, line width=3.5pt] (v5) -- (v13);
    \draw (v2) -- (v13);
    \draw (v5) -- (v13);
    
    \draw[red, line width=2pt] (v11) -- (v13);
    
    \end{tikzpicture}
    \par
}

Characterizing the internal edges in the prime graph complements of $\An{13}$-solvable groups is more difficult than in the $\An{11}$-solvable case, because more internal edges are possible. We use the results developed up to this point to obtain partial information about the internal edge structure in the $\An{13}$-solvable case.

\begin{lemma}
\label{lma: Brauer internal edges A_13}

    Let $H\in \{\An{13}, 2.\An{13}, \Sn{13}\}$, $r\in\pi(\An{13})$, and $p\in \{2,3,5,7\}$. Then
    \begin{enumerate}
        \item for every $q\in\{2,3,5,7,11\}$, there exists an element of order $q$ that has a fixed point in every representation corresponding to a character $\chi\in\IBr_r(H)$,
        \item if $\chi \in \IBr_r(H)$, and $p-13 \in A_{\chi}:=\left\{r-q\mid q\in\pi(\An{13}) \text{ s.t. }\exists g\in H,\ |g|=q,\ \frac{1}{|g|}\sum_{x\in\langle g\rangle}\chi(x)>0\right\}$, then $p-11\in A_{\chi}$ or $11-13\in A_{\chi}$, 
        \item there is $\chi_p\in \IBr_p(H)$ such that $A_{\chi_p}\cap \GK{\An{13}} = \{p-11\}$.
    \end{enumerate}
    
\end{lemma}
\begin{proof}

    First, we prove (1). If $q\in\{2,3,5\}$, then \cref{cor:Cr-leq-n/2-obstruction} gives the desired result. We now consider elements of order $7$ and $11$. Using GAP \cite{GAP}, we found that the dimension of the fixed point space, given by $\frac{1}{|g|}\sum_{x\in\langle g\rangle}\chi(x)$, is greater than zero for every $g\in \An{13}$ with $|g|\in\{7,11\}$ and every $\chi\in\IBr_r(H)$. 
    
    \medskip
    Consequently, if $p\in\{2,3,5,7\}$ and $\chi\in \IBr_p(H)$, then $11-p\in A_{\chi}$. The above results also guarantee that $11-13\in A_{\chi}$ for all $\chi \in \IBr_{13}(H)$. Moreover, if $p-13 \in A_{\chi}$ for some $\chi \in \IBr_{r}(H)$, then $r=p$ or $r=13$. Since $\chi\in \IBr_p(H)$ implies $p-11\in A_{\chi}$ and $\chi \in \IBr_{13}(H)$ implies $11-13\in A_{\chi}$, we obtain the desired result for (2). 
    
    \medskip
    Now, we show (3). \cref{cor:Ap} proves the existence of $\chi_p \in \IBr_p(H)$ such that all elements of order $13$ act fixed-point-freely, i.e. $p-13 \not\in A_{\chi_p}$. Hence,  $A_{\chi_p}\cap \GK{\An{13}} = \{p-11\}$, since $N_{\GK{\An{13}}}(p)=\{11,13\}$.
    
\end{proof}

\begin{lemma}
\label{lma: K-does-not-affect-H-internal-edges}
    Let $H$ and $K$ be finite groups, and let $B$ be a direct product of elementary abelian groups on which $H\times K$ acts. Suppose that $\pi(H)\cap\pi(K)=\emptyset$. Let $B_{\pi(H)}$ be the direct product of the elementary abelian factors of $B$ whose characteristic belongs to $\pi(H)$. Then
    \begin{equation*}
        \GK{B\rtimes(H\times K)}[\pi(H)] = \GK{B_{\pi(H)}\rtimes H}.
    \end{equation*}
\end{lemma}
\begin{proof}
    Immediate.
\end{proof}

\begin{lemma}
\label{A13 edge implication}
    Let $G$ be an $\An{13}$-solvable group such that $\GK{G}[\pi(\An{13})]$ contains a triangle. If $p\in\{2,3,5,7\}$ and $p-11\in\GK{G}$, then $p-13\in\GK{G}$.
\end{lemma}
\begin{proof}
    By \cite[Theorem 2.1.1]{alland2025realizinggruenbergkegelgraphstsolvable}, since $\GK{G}$ contains a triangle, $G$ satisfies the hypotheses of \cite[Theorem A]{alland2025realizinggruenbergkegelgraphstsolvable}. Hence $\GK{G}$ is realized by a group of the form $B\rtimes(H\times K)$, where $B$ is a direct product of elementary abelian groups, $K$ is solvable such that $\pi(K)\cap\pi(H)=\emptyset$, and either $H=2.\An{13}$ or $H=\An{13}.O$ for , possibly trivial, $O\leq\operatorname{Out}(\An{13})$. Since $\operatorname{Out}(\An{13})\cong C_2$, we have $H\in\{2.\An{13},\An{13},\Sn{13}\}$. In all cases, $\pi(H)=\pi(\An{13})=\{2,3,5,7,11,13\}$.
    \medskip
    
    By \cref{lma: K-does-not-affect-H-internal-edges}, the graph induced on $\pi(\An{13})$ is determined only by $B_{\pi(\An{13})}\rtimes H$. Therefore
    \begin{equation*}
        \GK{G}[\pi(\An{13})]\cong \GK{B_{\pi(\An{13})}\rtimes H}.
    \end{equation*}
    Combining this with \cref{A13 backwards direction}, we get
    \begin{equation*}
        \GK{G}[\pi(\An{13})]\cong
        \GK{H}\setminus\left(\bigcup_{r\in R}A_r\cup\{s-t\mid s,t\in R,\ s\neq t\}\right)
    \end{equation*}
    for some $R\subseteq\pi(H)$, where the sets $A_r$ are defined as in \cref{A13 backwards direction}. Also, for any two distinct $s,t \in R$, without loss of generality, we can assume that $s\ne 13$ and $s$ acts with fixed points in every characteristic $t$ irreducible Brauer representation, by \cref{lma: Brauer internal edges A_13} (1). Therefore, $s-t \in A_t \subseteq \bigcup_{r\in R}A_r$, and consequently, $\{s-t\mid s,t\in R,\ s\neq t\}\subseteq\bigcup_{r\in R}A_r$. This implies that 
    $$
    \GK{G}[\pi(\An{13})]\cong
        \GK{H}\setminus\left(\bigcup_{r\in R}A_r\right)
    $$
    
    Hence, every possible internal graph $\GK{G}[\pi(\An{13})]$ is obtained from the Brauer character data of one of the groups $2.\An{13}$, $\An{13}$, or $\Sn{13}$. In particular, if $p\in\{2,3,5,7\}$ then $p-13\notin\GK{G}$ if and only if there exists $\chi \in \IBr_p(H)\cup \IBr_{13}(H)$ such that $p-13\in A_{\chi} \subseteq \bigcup_{r\in R}A_r$, where $A_{\chi}$ is defined as in \cref{lma: Brauer internal edges A_13} (2). By the same result, if $p-13\in A_{\chi}$, then $p-11\in A_{\chi}$ or $11-13\in A_{\chi}$. However, the fact that $\GK{G}[\pi(\An{13})]$ contains a triangle guarantees that $11-13 \in \GK{G}[\pi(\An{13})]$, since this edge is contained in every triangle of $\GK{\An{13}}$. Thus, if $p-13\in A_{\chi}$, then $p-11\in A_{\chi}$, or equivalently, $p-11\in\GK{G}$ implies $p-13\in\GK{G}$.   
    
\end{proof}

\begin{theorem}
\label{A_13 Characterization}
    A graph $\Xi$ is the complement of the prime graph of an $\An{13}$-solvable group if and only if $\Xi$ is $3$-colorable and satisfies one of the following conditions:
    \begin{enumerate}
        \item $\Xi$ is triangle-free.
    
        \item $\Xi$ contains an induced subgraph $S$ on vertices
        $\{a,b,c_1,c_2,c_3,c_4\}$ satisfying the following properties:
        \begin{description}
            \item[(a)] If $v\in \Xi\setminus S$, $w\in S$, and $v-w\in \Xi$, then $w=a$.
    
            \item[(b)] The graph $\Xi\setminus S$ is triangle-free and has a $3$-coloring such that $N_\Xi(S)\setminus S$ is monochromatic.
    
            \item[(c)] We have $a-b\in S$, and $c_i-c_j\notin S$ for all distinct $c_i,c_j\in\mathcal{C}:=\{c_1,c_2,c_3,c_4\}$.
    
            \item[(d)] For every $c\in\mathcal{C}$, $c-b\in S$ implies that $c-a\in S$.
        \end{description}
    \end{enumerate}
\end{theorem}
\begin{proof}
    We first prove the forward implication. Suppose that $\Xi=\GK{G}$ for some $\An{13}$-solvable group $G$. If $G$ has no composition factor isomorphic to $\An{13}$, then $G$ is solvable, and hence $\Xi$ is triangle-free and $3$-colorable by \cref{thm: solvable case}. Thus, Condition~\textup{1} holds.
    
    We may therefore assume that $G$ is strictly $\An{13}$-solvable. Let $S=\Xi[\pi(\An{13})]$, and label its vertices so that $a=13$, $b=11$, and $\{c_1,c_2,c_3,c_4\}=\{2,3,5,7\}$. By \cref{thm: A_n complex rep,thm: 2.A_n complex rep}, if $p\in\pi(G)\setminus\pi(\An{13})$, $q\in\pi(\An{13})$, and $p-q\in \Xi$, then $q=13$. Therefore, $a$ is the only vertex of $S$ that may be adjacent to a vertex outside $S$. This proves Condition~\textup{2(a)}. Consequently, the hypotheses of \cite[Lemma 2.3.5] {keller2024classifyingprimegraphsfinite} are satisfied. It follows that $\Xi$ is $3$-colorable and admits a $3$-coloring in which $N_\Xi(S)\setminus S$ is monochromatic. The same result implies that every triangle of $\Xi$ is contained entirely in $S$ and $\Xi\setminus S$ is triangle-free. Hence Condition~\textup{2(b)} holds. Moreover, we may assume that $\Xi$ contains a triangle, since otherwise Condition~\textup{1} again holds. Since $\An{13}$ is a section of $G$, we have $S\subseteq\GK{\An{13}}$. In particular, no two vertices corresponding to primes in $\{2,3,5,7\}$ are adjacent in $S$. Thus, $c_i-c_j\notin S$ for all distinct $c_i,c_j\in\mathcal{C}$. Furthermore, every triangle in $\GK{\An{13}}$ contains the edge $11-13$. Since $\Xi$ contains at least one triangle and every triangle of $\Xi$ is contained in $S$, it follows that $a-b\in S$. This proves Condition~\textup{2(c)}. Finally, Condition~\textup{2(d)} follows from \cref{A13 edge implication}. This completes the proof of the forward implication.
    
    We now prove the reverse implication. Suppose first that $\Xi(X,E)$ is triangle-free. Since $\Xi$ is also $3$-colorable, \cref{thm: solvable case} guarantees the existence of a solvable group whose prime graph complement is isomorphic to $\Xi$. Since every solvable group is $\An{13}$-solvable, the result follows in this case.
    
    We may therefore assume that $\Xi$ contains an induced subgraph $S$ satisfying conditions~\textup{2(a)}--\textup{2(d)}. We begin by constructing an $\An{13}$-solvable group $G_2$ such that $\GK{G_2}\cong S$.
    
    Fix a bijection $f:S\to\pi(\An{13})$ such that $f(a)=13$, $f(b)=11$. Now define $\mathcal{C}_1:=\{c\in\mathcal{C}:c-b\notin S \text{ and }c-a\in S\}$. We show in \cref{lma: Brauer internal edges A_13} (3) that, for every $c\in\mathcal{C}_1$, there is a Brauer character of $\An{13}$ in characteristic $f(c)$, denoted by $\chi_c\in\IBr_{f(c)}(\An{13})$, satisfying
    $A_{\chi_c}\cap \GK{\An{13}} = \{f(c)-11\},$
    where
    \begin{equation*}
        A_{\chi_c}=\left\{f(c)-q \,\middle|\,
        \begin{array}{c}
            \text{there exists }g\in \An{13}\text{ such that }|g|=q
            \\
            \text{and }\displaystyle\frac{1}{|g|}\sum_{x\in\langle g\rangle}\chi_c(x)>0
        \end{array}
        \right\}.
    \end{equation*}
    
    We now apply \cref{lma: Brauer realization theorem}, set $R:=f(\mathcal{C}_1)$ and, for each $c\in\mathcal{C}_1$, let $Y_{f(c)}:=\{\chi_c\}$. For these choices, $A_{f(c)}=A_{\chi_c}$, and the additional edges between distinct primes in $R$ removed by the lemma are already absent from $\GK{\An{13}}$, since $R\subseteq\{2,3,5,7\}$. Therefore, \cref{lma: Brauer realization theorem} yields an $\An{13}$-solvable group $G_1$ such that $\GK{G_1}\cong\GK{\An{13}}\setminus\{f(c)-11:c\in\mathcal{C}_1\}$. 
    
    Next, define $\mathcal{C}_2:=\{c\in\mathcal{C}:c-a\notin S\}$, and let $G_2:=\left(\prod_{c\in\mathcal{C}_2}C_{f(c)}\right)\times G_1$. By Condition~\textup{(d)}, if $c\in\mathcal{C}_2$, then $c-b\notin S$. By Condition~\textup{(c)}, such a vertex $c$ is therefore isolated in $S$. On the other hand, the factor $C_{f(c)}$ is central in $G_2$, so $f(c)$ is isolated in $\GK{G_2}$.
    
    We claim that $\GK{G_2}\cong S$. Indeed, if $c\in\mathcal{C}_2$, then $f(c)$ is isolated in $\GK{G_2}$, and hence its neighborhood agrees with that of $c$ in $S$. If $c\in\mathcal{C}_1$, then the edge $f(c)-13$ remains in $\GK{G_2}$, whereas the construction removes the edge $f(c)-11$. Finally, if $c\notin\mathcal{C}_1\cup\mathcal{C}_2$, then both edges $f(c)-13$ and $f(c)-11$ remain in $\GK{G_2}$.
    
    Moreover, no edge occurs between $f(c_i)$ and $f(c_j)$ for distinct $i,j$, while the edge $11-13$ remains in $\GK{G_2}$. Therefore, under the labeling $f$, the graphs $\GK{G_2}$ and $S$ have the same vertex set and the same edge set. Hence, $\GK{G_2}\cong S$.
    
    We now show that for each vertex $v \in X \setminus S$, there exists a complex representation $\rho$ of $G_2$ such that for all $x \in S$, $x-v \in \Xi$ if and only if all elements of order $f(x)$ act without fixed points in $\rho$. In the case where $N(v) \cap S = \emptyset$, the trivial representation suffices. Consider then, the case where $v$ forms a bridging edge with some $x$. We know from 2(b) that if $x-v \in \Xi$, then $x=a$, which we associated to $13$ through $f$. By \cref{thm: A_n complex rep}, there exists a representation $\tilde{\rho}$ of $\An{13}$ in which only elements of order $13$ act without fixed points. Notice that $G_2 = N \rtimes \An{13}$ for $N=\left( \prod_{c\in\mathcal{C}_2} C_{f(c)} \right) \times V$ and $V$ from \cref{lma: Brauer realization theorem} such that $V \rtimes \An{13} = G_1$. Let $\pi: G_2 \to \An{13}$ be the natural projection map. We can then define a representation of $G_2$, $\rho:=\tilde{\rho} \circ \pi$. Since, by construction, $ \ker\rho = N$ and $13 \nmid |N|$, one can see that $|g|=13$ implies $|\rho(g)|=13$ for all $g \in G_2$. Consequently, $\rho$ is the desired representation of $G_2$. 
    
    Finally, we apply \cref{lma: psl2q realization} to extend the realization of $S$ to the entire graph $\Xi$ and conclude there exists an $\An{13}$-solvable group $G$ such that $\GK{G} \cong \Xi$.
\end{proof}

\section{Future Directions} 

As shown in \cref{Sec: Prime Graph Complement,Sec: Representations of An,Sec: Edges}, the tools developed in this paper may be useful in a future characterization of all graphs realizable as prime graph complements of $\mathcal A_n$-solvable groups. In \cref{Sec: A11 and A13}, we applied these tools to solve the cases $n=11$ and $n=13$. The approach used there could potentially be generalized to arbitrary $n$, provided that several remaining obstacles are addressed.

We first discuss the limitations of \cite[Theorem A]{alland2025realizinggruenbergkegelgraphstsolvable}. This result is used in \cref{A13 edge implication} to reduce the classification problem to the study of prime graph complements of groups isomorphic to $B\rtimes(H\times K)$, where $B$ is a direct product of elementary abelian groups, $K$ is solvable with $\pi(K)\cap\pi(H)=\emptyset$, and either $H=2.\An{13}$ or $H=\An{13}.O$ for some possible trivial $O\leq\operatorname{Out}(\An{13})$. However, the theorem applies only to groups satisfying one of the following conditions:
\begin{enumerate}
    \item The prime $2$ is adjacent to some $p\in \pi(G)\setminus \pi(T)$ in $\overline{\Gamma}(G)$.

    \item If $t\in \pi(T)$ is Fermat, then, for every odd $q\in \pi(G)$ such that $t$ is adjacent to $q$ in $\overline{\Gamma}(G)$, there exists a non-Fermat $r\in \pi(T)$ such that $r$ is adjacent to $q$ in $\overline{\Gamma}(G)$.
\end{enumerate}

In general, $\mathcal A_n$-solvable groups need not satisfy these conditions. For example, if $G\cong C_{11}\times C_7\times \mathcal A_{17}$, then $13$ is adjacent only to Fermat primes in $\GK{G}$. One possible way to remove this obstruction is suggested by the conjecture in \cite{alland2025realizinggruenbergkegelgraphstsolvable}, which predicts that the conclusion of Theorem A remains valid even without the Fermat-prime hypothesis.

A second limitation is that, for $n\geq 17$, the prime graph complement of $\mathcal A_n$ is not $3$-colorable. Indeed, \cref{chromatic number} shows that the chromatic number of $\GK{\mathcal A_n}$ grows with $n$. In the cases $n=11$ and $n=13$, we use the $3$-colorability of $\GK{\mathcal A_{11}}$ and $\GK{\mathcal A_{13}}$ to apply \cite[Lemma 2.3.5]{keller2024classifyingprimegraphsfinite}. A natural starting point for overcoming this issue is \cite[Lemma 1.2.1]{alland2025realizinggruenbergkegelgraphstsolvable}, which provides a related tool that does not rely on $3$-colorability. 

The final obstacle to extending our results to arbitrary $n$ is the lack of Brauer character tables for $\mathcal A_n$, $2.\mathcal A_n$, and $\mathcal S_n$ when $n\geq \text{16}$ \cite{GAP,ATLAS}. Such data is needed to determine which internal edges can be removed; for instance, we use this information in \cref{lma: Brauer internal edges A_13} for $\mathcal A_{13}$.

A first step toward addressing this limitation appears already in \cref{lma: Brauer internal edges A_13} itself, where we use the tools developed in \cref{Sec: Representations of An} to obtain fixed-point-free information without relying on a full Brauer character table. Generalizing the techniques introduced in \cref{cor:Cr-leq-n/2-obstruction} and \cref{cor:Ap} seems to be the most plausible path forward, since computing Brauer character tables for large $n$ is currently not feasible.

Motivated by these observations, we formulate a series of questions concerning elements that do not act fixed-point-freely on irreducible Brauer modules for $\mathcal{A}_n$, $2.\mathcal{A}_n$, and $\mathcal{S}_n$. Positive answers to these questions would further clarify the relationship between element orders, conjugacy classes, and the existence of fixed points in modular representations.

\begin{question}
    Let $n \geq 10$, and let $p$ and $q$ be distinct primes such that $p,q \leq n$ and $n/2<p<n-1$. For every irreducible $\F_q[\A_n]$-module, does there exist an element $g\in \A_n$ of order $p$ such that $g$ has a nonzero fixed point on the module?
\end{question}

\noindent
This is precisely the range in which all elements of order $p$ form a single conjugacy class in $\A_n$, which also implies that the Brauer character values of these elements are integers and follow directly from $\mathcal{S}_n$. Consequently, a positive answer would imply that every element of order $p$ has a nonzero fixed point on every irreducible $\F_q[\A_n]$-module.

\begin{question}
\label{Q:does_not_divide}
    Let $p$ and $q$ be distinct primes with $p,q\leq n$, and suppose that $p=n-1$ and $q\nmid n$. For every irreducible $\F_q[\A_n]$-module, does there exist an element $g\in \A_n$ of order $p$ such that $g$ has a nonzero fixed point on the module?
\end{question}

\noindent In \cref{Sec: Representations of An} we answer the complement of \cref{Q:does_not_divide} for when $q \mid n$. \cref{cor:Ap+1_q|p} serves in constructing an irreducible module in which elements of order $p=n-1$ act without fixed points. The conclusion of the above Question would suffice in understanding whether, outside of the exceptional case of \cref{cor:Ap+1_q|p}, Brauer fixed-point freeness follows directly from the complex case. 

\begin{question}
    Let $p$ and $q$ be distinct primes with $p,q\leq n$. Does there exist an element $g\in 2.\A_n$ of order $p$ such that $g$ has a nonzero fixed point on every irreducible spin $\F_q[2.\A_n]$-module?
\end{question}

We believe that progress in each of the directions discussed above---extending the structural results used in our classification, replacing the reliance on 3-colorability with tools applicable to larger alternating groups, and obtaining the necessary fixed-point information without requiring complete Brauer character tables---would allow the framework developed in this paper to be extended to general $n$. In particular, answers to Questions 5.1--5.3 would address much of the remaining representation-theoretic difficulty and, together with suitable advances in the other two directions, could ultimately lead to a characterization of the prime graph complements of $\A_n$-solvable groups for arbitrary $n$.

\section*{}
\begin{center}
    \textbf{Acknowledgements}
\end{center}
 This research is the result of an REU project at Texas State University conducted under the mentorship of T. M. Keller. A. Beun, E. Cabrera, D. N. Goto, and T. M. Keller were supported by NSF-REU grant DMS-2447229 as well as Jane Street. L. Abounader, A. Bell, and T. M. Keller were supported by a Texas State REP grant. The authors gratefully acknowledge the financial support of NSF and Jane Street, and thank Texas State University for providing a great working environment and support. T. M. Keller was also partially supported by a grant from the Simons Foundation (SFI-MPS-TSM-00014153 to T. M. K.).




\begin{thebibliography}{99}

\bibitem{ATLAS}
R. Abbott, J. Bray, S. Linton, S. Nickerson, S. Norton, R. Parker, I. Suleiman, J. Tripp, P. Walsh, and R. Wilson. \emph{Atlas of finite group representations --- version 3}. \url{https://brauer.maths.qmul.ac.uk/Atlas/v3/}.

\bibitem{alland2025realizinggruenbergkegelgraphstsolvable}
L. Alland, A. Fridman, and T. M. Keller. ``Realizing Gruenberg--Kegel graphs of $T$-solvable groups with structurally simplified extensions of $T$''. arXiv preprint. 2025. arXiv: 2511.16403. URL: \url{https://arxiv.org/abs/2511.16403}.

\bibitem{amrutha2024cyclic}
P. Amrutha, A. Prasad, and Velmurugan S. ``Cyclic characters of alternating groups''. arXiv preprint. 2024. arXiv: 2403.05109. URL: \url{https://arxiv.org/abs/2403.05109}.

\bibitem{Cliff}
A. Clifford. ``Representations induced in an invariant subgroup''. In: \emph{Ann. of Math. (2)} 38.3 (1937), pp. 533--550. DOI: 10.2307/1968599.

\bibitem{edwards2022primegraphsgroupsarithmetically}
T. J. Edwards, T. M. Keller, R. M. Pesak, and K. Sellakumaran Latha. ``The prime graphs of groups with arithmetically small composition factors''. In: \emph{Ann. Mat. Pura Appl. (4)} 203.2 (2024), pp. 945--973.

\bibitem{florez2022primegraphs}
C. Flores, J. Higgins, K. Huang, T. M. Keller, D. Shen, and Y. Yang. ``The prime graphs of some classes of finite groups''. In: \emph{J. Pure Appl. Algebra} 226.7 (2022), Paper No. 106990, 19pp.

\bibitem{GRUBER2015397}
A. Gruber, T. M. Keller, M. Lewis, K. Naughton, and B. Strasser. ``A characterization of the prime graphs of solvable groups''. In: \emph{J. Algebra} 442 (2015), pp. 397--422.

\bibitem{HuangKellerKissingerPlotnickRomaYang+2024+89+117}
Z. Huang, T. M. Keller, S. Kissinger, W. Plotnick, M. Roma, and Y. Yang. ``A classification of the prime graphs of pseudo-solvable groups''. In: \emph{J. Group Theory} 27.1 (2024), pp. 89--117.

\bibitem{isaacs_finite_2008}
I. M. Isaacs. \emph{Finite Group Theory}. Graduate Studies in Mathematics. American Mathematical Society, 2008.

\bibitem{james1978}
G. D. James. \emph{The Representation Theory of the Symmetric Groups}. Vol. 682. Lecture Notes in Mathematics. Berlin: Springer-Verlag, 1978.

\bibitem{keller2024classificationprimegraphsoperatornamesz8}
T. M. Keller, Z. Martin, A. Renner, G. Roca, and E. Yu. ``Classification of the prime graphs of $\operatorname{Sz}(8)$-, $\operatorname{Sz}(32)$-, and $\operatorname{PSL}(2, 2^5)$-solvable groups''. submitted. 2024. arXiv: 2410.21063 [math.GR].

\bibitem{keller2025criteriaclassifyingprimegraphsPSL2q}
T. M. Keller, Z. Martin, A. Renner, G. Roca, and E. Yu. ``Criteria for classifying prime graphs of PSL(2, $q$)-solvable groups''. submitted. 2025. arXiv: 2510.21979 [math.GR].

\bibitem{keller2024classifyingprimegraphsfinite}
T. M. Keller, G. Pettigrew, S. Solotko, and L. Zheng. ``Classifying prime graphs of finite groups -- a methodical approach''. In: \emph{J. Pure Appl. Algebra} 229.11 (2025), 108089, 40pp.

\bibitem{MahadevPeled}
N. V. R. Mahadev and U. N. Peled. \emph{Threshold Graphs and Related Topics}. Vol. 56. Annals of Discrete Mathematics. North-Holland, 1995.

\bibitem{prasad2025minimal}
A. Prasad, A. Staroletov, and Velmurugan S. ``Minimal polynomials in spin representations of symmetric and alternating groups''. submitted. 2025. arXiv: 2512.24158 [math.GR].

\bibitem{GAP}
The GAP Group. \emph{GAP -- Groups, Algorithms, and Programming, Version 4.14.0}. 2024.

\bibitem{VasilievVdovin2005}
A. V. Vasil'ev and E. P. Vdovin. ``An Adjacency Criterion for the Prime Graph of a Finite Simple Group''. In: \emph{Algebra and Logic} 44.6 (2005), pp. 381--406. DOI: 10.1007/s10469-005-0037-5.

\bibitem{wolf2011spaces}
J. A. Wolf. \emph{Spaces of Constant Curvature}. 6th. Providence, RI: AMS Chelsea Publishing, 2011.

\end{thebibliography}
\end{document}